\documentclass[final,twoside,11pt]{entics}
\usepackage{enticsmacro}
\usepackage{tikz}
\usepackage{graphicx}
\usepackage[all]{xy}
\usepackage{float}

\def\conf{ISDT 2022} 	
\volume{2}			
\def\volu{2}			
\def\papno{9}			
\def\mydoi{{\normalshape  \href{https://doi.org/10.46298/entics.10355}{10.46298/entics.10355}}}
\def\copynames{X. Ruan, X. Xu} 
\def\runtitle{$SI_{2}$-quasicontinuous spaces}

\def\lastname{Ruan and Xu}
\begin{document}

\begin{frontmatter}
  \title{$SI_{2}$-quasicontinuous Spaces\thanksref{ALL}}
  \author{Xiaojun Ruan\thanksref{a}\thanksref{myemail}}	
   \author{Xiaoquan Xu\thanksref{b}\thanksref{coemail}}		
   \address[a]{Department of Mathematics, Nanchang University, Nanchang 330031, China}  	
   \address[b]{School of Mathematics and Statistics, Minnan Normal University, Zhangzhou 363000, China}
   						
  \thanks[ALL]{Supported by the National Natural Science Foundation of
China (11661057, 12061046, 12071199, 61866023), GAN PO555 Program of Jiangxi Province and the Natural Science Foundation of
 of Jiangxi Province (20192ACBL20045).}   
   \thanks[myemail]{Email: \href{mailto:myuserid@mydept.myinst.myedu} {\texttt{\normalshape
        rxj54188@163.com}}}

\thanks[coemail]{Email: \href{mailto:couserid@codept.coinst.coedu} {\texttt{\normalshape
        xiqxu2002@163.com}}}

\begin{abstract}
 In this paper, as a common generalization of $SI_{2}$-continuous spaces and $s_{2}$-quasicontinuous posets, we introduce the concepts of $SI_{2}$-quasicontinuous spaces and $\mathcal{GD}$-convergence of nets for arbitrary topological spaces by the cuts. Some characterizations of $SI_{2}$-quasicontinuity of spaces are given. The main results are: (1) a space is $SI_{2}$-quasicontinuous if and only if its weakly irreducible topology is hypercontinuous under inclusion order; (2) A $T_{0}$ space $X$ is $SI_{2}$-quasicontinuous if and only if the $\mathcal{GD}$-convergence in $X$ is topological.
\end{abstract}
\end{frontmatter}
\section{Introduction}
The theory of continuous lattices and domains which arose from computer science and logic, is based on the investigation of directed complete posets (dcpos, for short) in which every directed set has a least upper bound. Since there are important mathematical models where arise posets such as the reals $\mathbb{R}$  and the naturals $\mathbb{N}$ which fail to be dcpos, there is an attempt to carry as much as possible of the theory of continuous
lattices to as general an ordered structure as possible (see  \cite{LawsonXu2004,HLL2009,M1999,M1981,XX2009,PV1990,Zhang1993,Zhang2015}). Moreover, Ern$\acute{e}$ pointed out the importance of the concept of \emph{standard completions} \cite{E1981,E1991,E1999} in the context of generalized continuous posets. In the absence of enough joins, Ern$\acute{e}$ introduced the concept of $s_2$-\emph{continuous} posets and the \emph{weak Scott topology} by means of the cuts instead of joins (see \cite{E1981}), whence they easily apply to generalized settings where certain joins are missing. Quasicontinuous domains introduced by Gierz, Lawson and Stralka \cite{GJA1983} capture many of the essential features of continuous domains  and pop up from time to time generalizing slightly the powerful notion of continuous domains. Recently they have attracted increased attention through the remarkable work of J. Goubault-Larrecq (see \cite{Goubault2012}). As a generalization of $s_2$-continuity, the concept of $s_2$-quasicontinuity was introduced in \cite{Zhang2015} by Zhang and Xu, their basic idea is to generalize the way-below relation between the points to the case of sets. It was proved that $s_2$-quasicontinuous posets equipped with the weak Scott topologies are precisely the hypercontinuous lattices. It is well known that given a topological space one can also define ordered structures (see \cite{E1991,GHKLMS2003,E2009,Zhao2015}). At the 6th International Symposium in Domain Theory, J.D. Lawson emphasized the need to develop the core of domain theory directly in $T_{0}$ topological spaces instead of posets. Moreover, it was pointed out that several results in domain theory can be lifted from the context of posets to $T_{0}$ topological spaces (see~\cite{E2009,Zhao2015,Andradi2018,Andradi2019}). In \cite{E2009}, Ern$\acute{e}$ further proved that the weak Scott topology is the weakest monotone determined topology with a given specialization order. In \cite{GHKLMS2003}, the concept of $\mathcal{S}$-convergence for dcpos was introduced by Scott to characterize continuous domains. It was proved that for a dcpo, the $\mathcal{S}$-convergence is topological if and only if it is a continuous domain. The strong interplay between order theory and topology is often exhibited by the use of different types of convergence; in particular, various notions of continuity of ordered structures can be characterized in this way. In \cite{Zhao2015}, Zhao and Ho defined a new way-below relation and a new topology constructed from any given topology on a set using irreducible sets in a $T_{0}$ topological space replacing directed subsets and investigated the properties of this derived topology and $k$-bounded spaces. It was proved that a space $X$ is $SI$-continuous if and only if $SI(X)$ is a $C$-space.

In this paper, we introduce the concepts of $SI_{2}$-quasicontinuous spaces and $\mathcal{GD}$-convergence of nets for arbitrary topological spaces by the cuts. Some characterizations of $SI_{2}$-quasicontinuity of spaces are given. We show that a space is $SI_{2}$-quasicontinuous if and only if its weakly irreducible topology is hypercontinuous under inclusion order. Finally we arrive at the conclusion that a $T_{0}$ space $X$ is $SI_{2}$-quasicontinuous if and only if the $\mathcal{GD}$-convergence in $X$ is topological.

\section{Preliminaries}

Let $P$ be a partially ordered set (poset, for short). A nonempty set $D\subseteq P$ is \emph{directed} if for any $d_{1}, d_{2}$ in $D$ there exists $d$ in $D$ above $d_{1}$ and $d_{2}$. For $A\subseteq P$, $\uparrow A=\{x\in P: a\leq x$ for some $a\in A\}$. A subset $A$ of $P$ is an \emph{upper set} if $A=\uparrow A$. The order dual concepts are $\downarrow A$ and \emph{lower set}. $A^\uparrow$ and $A^\downarrow$ denote the sets of all \emph{upper} and \emph{lower bounds} of $A$, respectively. A \emph{cut} $\delta$ of $A$ in $P$ is defined by $A^{\delta} =(A^\uparrow)^\downarrow$. We say that a nonempty family $\mathcal{G}$ of sets is \emph{directed} if given $G_1$, $G_2$ in the family $\mathcal{G}$, there exists $G\in\mathcal{G}$ such that $G_1$, $G_2\leq G$, i.e., $G\subseteq\uparrow\! G_1\cap\uparrow\! G_2$. For nonempty subsets $F$ and $G$ of $P$, we say $F$ \emph{approximates} $G$ if for every directed subset $D\subseteq P$, whenever $\bigvee D$ exists, $\bigvee D\in\uparrow\! G$ implies $d\in\uparrow\! F$ for some $d\in D$. A dcpo $P$ is called a \emph{quasicontinuous domain} if for all $x\in P$, $\uparrow\! x$ is the directed (with respect to reverse inclusion) intersection of sets of the form $\uparrow\! F$, where $F$ approximates $\{x\}$ and $F$ is a finite set of $P$. Let $P^{(<\omega)}$ be the set of all nonempty finite subsets of $P$.

Given a poset $P$, we can generate some intrinsic topologies. The upper sets form the \emph{Alexandroff topology} $\alpha(P)$. For all directed sets $D$ of $P$, the \emph{weak Scott topology} $\sigma_2(P)$ consists of all upper sets $U$ such that $D^\delta\cap U\neq\emptyset$ implies $D\cap U\neq\emptyset$. Obviously if $P$ is a dcpo, then the weak Scott topology coincides with the usual \emph{Scott topology} $\sigma(P)$, which consists of all upper sets $U$ such that $\bigvee D\in U$ implies $D\cap U\neq\emptyset$ for all directed sets $D$ in $P$. For $x\in P$, the topology
generated by the collection of sets $P\mathop{\backslash}\downarrow \! x$ (as a subbase) is called the \emph{upper topology} and is denoted by $\upsilon(P)$. It is easy to see that $\upsilon(P)\subseteq\sigma_2(P)\subseteq\sigma(P)\subseteq\alpha(P)$.

For a topological space $(X, \tau)$, the \emph{specialization order} $\leq$ on $X$ is defined by $y\leq x$ if and only if $y\in$ $\operatorname{cl} (\{x\})$. If $(X, \tau)$ is $T_{0}$ then the specialization order is a partial order. The specialization order on $X$ is denoted by $\leq_{\tau}$ if there is need to emphasize the topology $\tau$. Keep in mind that the specialization order of the Alexandroff topology on a poset coincides with the underlying order.

Let $(X, \tau)$ be a topological space. A nonempty subset $F\subseteq X$ is called \emph{irreducible} if for every closed sets $B$ and $C$, whenever $F\subseteq B\bigcup C$, one has either $F\subseteq B$ or $F\subseteq C$. The set of all irreducible sets of the topological space $(X, \tau)$ will be denoted by Irr$_{\tau}(X)$ or Irr$(X)$.

Unless otherwise stated, in the context of $T_0$ spaces, all order-theoretic concepts such as lower sets, upper sets, etc, are taken with respect to the specialization order of the underlying spaces.
\begin{lemma}\label{remark-2.1}
Let $(X, \tau)$ be a $T_{0}$ space.
\begin{enumerate}
\item[\rm(1)] If $D\subseteq X$ is a directed set with respect to the specialization order, then $D$ is irreducible;
\item[\rm(2)] If $U\subseteq X$ is an open set, then $U$ is an upper set; Similarly, if $F\subseteq X$ is a closed set, then $F$ is a lower set.
\end{enumerate}
\end{lemma}

\begin{definition}{\rm(\cite{E1981,Zhang2015})}\label{definition-2.1}
Let $P$ be a poset and $x\in P$, $A, B\subseteq P$.
\begin{enumerate}
\item[\rm(1)]We say that $A$ is $way$ $below$ $B$ and write $A\ll B$ if for all directed sets $D\subseteq P$, $B\cap D^\delta\neq\emptyset$ implies $A\cap D\neq\emptyset$. We write $A\ll x$ for $A\ll\{x\}$ and $y\ll B$ for $\{y\}\ll B$. Let $w(x)=\{F\in P^{(<\omega)}:F\ll x\}$, $\Downarrow x=\{y\in P:y\ll x\}$.
\item[\rm(2)]$P$ is called $s_2$-$continuous$ if for all $x\in P$, $x\in (\Downarrow\! x)^\delta$ and $\Downarrow\! x$ is directed.
\item[\rm(3)]$P$ is called $s_2$-$quasicontinuous$ if for each $x\in P$, $\uparrow\! x=\bigcap\{\uparrow\! F:F$ is finite, $F\ll x$ \} and $w(x)=\{\uparrow F\subseteq P:F$ is finite and $F\ll x\}$ is directed.

\end{enumerate}
\end{definition}

\begin{definition}(\cite{Zhao2015})\label{definition-2.2}
Let $(X, \tau)$ be a $T_{0}$ space. For $x, y\in X$, define $x\ll_{SI} y$ if for all irreducible sets $F$, $y\leq \bigvee F$ implies there exists $e\in F$ such that $x\leq e$ whenever $\bigvee F$ exists. The set $\{y\in X: y\ll_{SI} x\}$ is denoted by $\Downarrow_{SI} x$ and the set $\{y\in X: x\ll_{SI} y\}$ by $\Uparrow_{SI} x$.
\end{definition}

\begin{definition}(\cite{Luo2019,Ruan2019})
Let $(X, \tau)$ be a $T_{0}$ space and $x, y\in X$. Define $x\ll_{r} y$ if for every irreducible set $E$, $y\in E^{\delta}$ implies there exists $e\in E$ such that $x\leq e$. We denote the set $\{y\in X: y\ll_{r} x\}$ by $\Downarrow_{r} x$ and the set $\{y\in X: x\ll_{r} y\}$ by $\Uparrow_{r} x$.
\end{definition}
\begin{definition}(\cite{Ruan2019})\label{definition-2.3}
Let $(X, \tau)$ be a $T_{0}$ space. $X$ is called $SI_{2}$-continuous if the following conditions are satisfied:
\begin{enumerate}
\item[\rm(1)] $\Downarrow_{r} x$ is directed for all $x\in X$;
\item[\rm(2)] $x\in (\Downarrow_{r} x)^{\delta}$ for all $x\in X$;
\item[\rm(3)] $\Uparrow_{r} x\in \tau$ for all $x\in X$.
\end{enumerate}
\end{definition}

\begin{definition}(\cite{Luo2019,Ruan2019})\label{definition-2.4}
Let $(X, \tau)$ be a $T_{0}$ space. A subset $U\subseteq X$ is called weakly irreducibly open if the following conditions are satisfied:
\begin{enumerate}
\item[\rm(1)] $U\in \tau$;
\item[\rm(2)] $F^{\delta}\bigcap U\neq\emptyset$ implies $F\bigcap U\neq\emptyset$ for all $F\in$Irr$_{\tau}(X)$.
\end{enumerate}
It is easy to check that the set of all weakly irreducibly open sets of $(X, \tau)$ is a topology, which will be called weakly irreducible topology of $X$ and will be denoted by $\tau_{SI_{2}}(X)$. $SI_2(X,\tau)=(X, \tau_{SI_{2}}(X))$ denotes the weakly irreducible topological space.
\end{definition}

The following lemma is the well-known Rudin Lemma.

\begin{lemma}{\rm(\cite{GHKLMS2003})}\label{lemma-2.1}\label{lem:Rudin}
Let ${\cal F}$ be a directed family of nonempty finite subsets of a poset $P$. Then there exists a directed set $D\subseteq \bigcup\limits_{F\in\mathcal{F}}F$ such that $D\cap F\neq \emptyset$ for all $F\in \mathcal{F}$.
\end{lemma}
\begin{definition}(\cite{Ruan2019})
Let $(X, \tau)$ be a $T_{0}$ space. A net $(x_{j})_{j\in J}$ in $X$ is said to converge to $x\in X$ if there exists a directed set $D\subseteq X$ with respect to the specialization order such that
\begin{enumerate}
\item[\rm(1)] $x\in D^{\delta}$;
\item[\rm(2)] For all $d\in D$, $d\leq x_{j}$ holds eventually.
\end{enumerate}
In this case we write $x\equiv_{\mathcal{D}}$$\lim$ $x_{j}$.
\end{definition}

\section{SI$_2$-quasicontinuous spaces}

In this section, the concept of $SI_2$-quasicontinuous spaces is introduced. Some properties of $SI_2$-quasicontinuous spaces are discussed.

\begin{definition}(\cite{Luo2019})\label{definition-3.1}
Let $X$ be a $T_{0}$ space and $x\in X$, $A, B\subseteq X$. We say that $A$ is $way$ $below$ $B$ and write $A\ll_{r} B$ if for all irreducible sets $E\subseteq X$, $B\cap E^\delta\neq\emptyset$ implies $A\cap E\neq\emptyset$. We write $A\ll_{r} x$ for $A\ll_{r}\{x\}$ and $y\ll_{r} B$ for $\{y\}\ll_{r} B$. The set $\{x\in X: A\ll_{r} x\}$ is denoted by $\Uparrow_{r} A$.
\end{definition}

\begin{definition}\label{definition-3.2}
A $T_{0}$ space $(X, \tau)$ is called $SI_2$-$quasicontinuous$ if for all $x\in X$, the following conditions are satisfied:
\begin{enumerate}
\item[\rm(1)]$w(x)=\{F\subseteq X:F\in X^{(<\omega)}$ and $F\ll_{r} x\}$ is directed;
\item[\rm(2)]$\uparrow\! x=\bigcap\{\uparrow\! F:F\in w(x)$\};
\item[\rm(3)]For any $H\in X^{(<\omega)}$, $\Uparrow_{r}H\in \tau$.
\end{enumerate}
\end{definition}

The following proposition is simple and the proof is omitted.

\begin{proposition}\label{proposition-3.1}
Let $X$ be a $T_{0}$ space and $G$, $H$, $K$, $M\subseteq X$. Then
\begin{enumerate}
\item[\emph{(1)}] $G\ll_{r} H\Leftrightarrow G\ll_{r} h$ for all $h\in H$;
\item[\emph{(2)}] $G\ll_{r} H\Leftrightarrow \uparrow G\ll_{r} \uparrow H$;
\item[\emph{(3)}] $G\ll_{r} H\Rightarrow G\leq H$, that is, $\uparrow H\subseteq\uparrow G$. Hence $\uparrow H\subseteq\bigcap \{\uparrow F: F\in w(H)\}$.
\item[\emph{(4)}] $G\leq H\ll_{r} K\leq M\Rightarrow G\ll_{r} M$.
\end{enumerate}
\end{proposition}
Clearly, an $SI_2$-continuous space must be $SI_2$-quasicontinuous, but the converse may not be true.
\begin{example}\label{example-3.1}
Let $P=\{a\}\cup\{a_n:n\in \mathbb{N}\}$, where $\mathbb{N}$ denotes the set of all positive integers. The partial order on $P$ is defined by setting $a_n\leq a_{n+1}$ for all $n\in \mathbb{N}$, and $a_1\leq a$. Then $P$ endowed with the Alexandroff upper topology is an $SI_2$-quasicontinuous space which is not $SI_2$-continuous.
\end{example}

\begin{lemma}
Let $P$ be a poset. Then $SI_{2}(P, \alpha(P))=(P, \sigma_{2}(P))$.
\end{lemma}

\begin{remark}
\begin{enumerate}
\item[\rm(1)] Let $P$ be a poset. Then $P$ is an $s_{2}$-quasicontinuous poset if and only if it is an $SI_{2}$-quasicontinuous space with respect to the Alexandroff topology.
\item[\rm(2)] Let $(X, \tau)$ be a $T_{0}$ space. If $X$ is an $SI_{2}$-quasicontinuous space, then it is also an $s_{2}$-quasicontinuous poset under the specialization order. But the converse may not be true.
\end{enumerate}
\end{remark}

\begin{example}\label{example-3.1}
Let $X$ be an infinite set with a cofinite topology $\tau$. Then it is a $T_{1}$ space. Clearly it is an antichain under the specialization order, and hence it is an $s_{2}$-continuous poset. Thus it is also $s_{2}$-quasicontinuous. But $\Uparrow_{r} x=\{x\}\not\in\tau$ for all $x\in X$, then $(X, \tau)$ is not an $SI_{2}$-quasicontinuous space.
\end{example}

\begin{lemma}[\cite{Luo2019}]\label{lemma-3.2}
Let $\mathcal{F}$ be a directed family of nonempty finite sets in a $T_{0}$ space $X$. If $G\ll_{r} x$ with $\bigcap\mathcal{F}\subseteq \uparrow x$ then there exists some $F\in \mathcal{F}$ such that
$F\subseteq \uparrow G$.
\end{lemma}

We now derive the interpolation property of $SI_2$-quasicontinuous spaces.
\begin{theorem}\label{theorem-3.1}
Let $X$ be an $SI_2$-quasicontinuous space. Then
\begin{enumerate}
\item[\rm(1)]For $x\in X, F\ll_{r} x$ implies that there exists some $G \in w(x)$ such that $F\ll_{r} G\ll_{r} x$.
\item[\rm(2)]For any $F, H\in X^{(<\omega)}, F\ll_{r} H$ implies that there exists some $G \in X^{(<\omega)}$ such that $F\ll_{r} G\ll_{r} H$.
\item[\rm(3)]For all $F\in X^{(<\omega)}, E\in Irr(X), F\ll_{r} E^{\delta}$ implies that there exists some $e\in E$ with $F\ll_{r} e$.
\end{enumerate}
\end{theorem}
\begin{proof}
(1) The statement has been proved for quasicontinuous domains in [3, Proposition III-3.5], and the similar proof carries over  to this setting.

(2) For all $H\in P^{(<\omega)}$, $F\subseteq X$, if $F\ll_{r} H$, then for all $h\in H$, $F\ll_{r} h$ by
Proposition \ref{proposition-3.1}.
 By (1), there exists a finite set $G_h$ such that $F\ll_{r} G_h\ll_{Z} h$. Let $G=\bigcup\limits_{h\in H}G_h$.
  Then $G\in P^{(<\omega)}$
 and $F\ll_{r} G\ll_{r} H$.

(3) For all $F\in X^{(<\omega)}, E\in Irr(X)$, if $F\ll_{r} E^{\delta}$, then there exist $G, H \in X^{(<\omega)}$
such that
$F\ll_{r} G\ll_{r} H\ll_{r} E^{\delta}$. Hence there are $g\in G, e\in E$ with $g\leq e$. Clearly, $F\ll_{r} g$, and then
 $F\ll_{r} e$.
\end{proof}

\begin{proposition}\label{proposition-3.2}
Let $(X, \tau)$ be an $SI_2$-quasicontinuous space.
\begin{enumerate}
\item[\emph{(1)}] For any $H\in X^{(<\omega)}$, $\Uparrow_{r} H$=$\operatorname{int}_{\tau_{SI_2(X)}}\uparrow\! H$.

\item[\emph{(2)}] A subset $U$ of $X$ is weakly irreducibly open iff for each $x\in U$ there exists a finite $F\ll_{r} x$ such that $\uparrow\! F\subseteq U$.

\item[\emph{(3)}] The sets $\{\Uparrow_{r} F:F\in X^{(<\omega)}\}$ form a basis for $\tau_{SI_2(X)}$.
\end{enumerate}
\end{proposition}
\begin{proof}
(1) Suppose that $H$ is a nonempty set in $X$. It is easy to show that $int_{\tau_{SI_2(X)}}\uparrow\! H\subseteq \Uparrow_{r}\! H$. Conversely, we show that $\Uparrow_{r}\! H\in\tau_{SI_2(X)}$. Obviously, $\Uparrow_{r} H\in\tau$. For all irreducible sets $E\subseteq X$, if $E^\delta\cap\Uparrow_{r}\! H\neq\emptyset$, i.e., there exists $e\in E^\delta\cap\Uparrow_{r}\! H$, then there exists a $F\in X^{(<\omega)}$ such that $H\ll_{r} F\ll_{r} e$ by Theorem \ref{theorem-3.1}. Since $e\in E^\delta$, $\uparrow\! F\cap E\neq\emptyset$, i.e., there exists some $y\in\uparrow F\cap E$. Thus $y\in\Uparrow_{r}\! H\cap E\neq\emptyset$ by Proposition \ref{proposition-3.1}. So $\Uparrow_{r} H\in\tau_{SI_2(X)}$. Since $\Uparrow_{r} H\subseteq\uparrow\! H$, $\Uparrow_{r} H\subseteq int_{\tau_{SI_2(X)}}\uparrow\! H$.

(2) Let $U\in\tau_{SI_2(X)}$, $x\in U$. By the definition of the weakly irreducibly open set, we have $U\ll_{r} x$, so by Theorem \ref{theorem-3.1}, there exists $F\in X^{(<\omega)}$ such that $U\ll_{r} F\ll_{r} x$. Thus $\uparrow\! F\subseteq U$. Conversely, suppose that the condition is satisfied, we have $U=\cup\{\Uparrow_{r} F: F\in X^{(<\omega)}, \uparrow F\subseteq U\}$. Since the space $X$ is $SI_2$-quasicontinuous, so $U\in \tau$. For all irreducible sets $E\subseteq X$, if $U\cap E^\delta\neq\emptyset$, then there exists $y\in U\cap E^\delta$. By hypothesis, there exists a finite $G\ll_{r} y$ such that $\uparrow\! G\subseteq U$. Therefore, $E\cap U\supseteq E\cap\uparrow\! G\neq\emptyset$. Thus $U$ is weakly irreducibly open.

(3) Let $U\in\tau_{SI_2(X)}$. By (2), we have $U=\cup\{\Uparrow_{r} F: F\in X^{(<\omega)}, \uparrow F\subseteq U\}$. By (1), $\Uparrow_{r}\! F\in\tau_{SI_2(X)}$ for all $F\in X^{(<\omega)}$. Thus the sets $\{\Uparrow_{r}\! F:F\in X^{(<\omega)}\}$ form a basis for $\tau_{SI_2(X)}$.
\end{proof}

\begin{definition}(\cite{GierzLawson1981})\label{definition-3.9}
For a complete lattice $L$, define a relation $\prec$ on $L$ by $x\prec y$ $\Leftrightarrow$ $y\in\mbox{int}_{\upsilon(L)}\uparrow\! x$. $L$ is called hypercontinuous if $x=\bigvee\{u\in L :u\prec x\}$ for
all $x\in L$.
\end{definition}

Now we give the topological characterizations of $SI_2$-quasicontinuous spaces.

\begin{theorem}\label{theorem-3.2} For a $T_{0}$ space $X$, the following statements are equivalent:
\begin{enumerate}
\item[\emph{(1)}] $X$ is an $SI_2$-quasicontinuous space;

\item[\emph{(2)}] For all $x\in X$ and $U\in\tau_{SI_{2}}(X)$, $x\in U$ implies that there exists $F\in X^{(<\omega)}$ such that $x\in int_{\tau_{SI_{2}}(X)}\uparrow F\subseteq\uparrow F\subseteq U$;

\item[\emph{(3)}] $(\tau_{SI_{2}}(X)),\subseteq)$ is a hypercontinuous lattice.
\end{enumerate}
\end{theorem}
\begin{proof}
$(1)\Rightarrow(2)$: It follows from Proposition \ref{proposition-3.2}.

$(2)\Rightarrow(1)$: For all $x\in P$, let $\mathcal{F}=\{F\in P^{(<\omega)}: x\in int_{\tau_{SI_{2}}(X)}\uparrow\! F\}$. Since $X\in\tau_{SI_{2}}(X)$, it follows from (2) that there exists $F\in X^{(<\omega)}$ such that $x\in int_{\tau_{SI_{2}}(X)}\uparrow\! F\subseteq\uparrow\! F\subseteq X$. Then $F\in\mathcal{F}\neq\emptyset$.  Obviously $\uparrow\! x\subseteq\bigcap\limits_{F\in\mathcal{F}}\uparrow\! F$. If $x\nleq y$, then $x\in X\mathop{\backslash}\downarrow\! y\in\tau_{SI_{2}}(X)$. By (2), there exists $F\in X^{(<\omega)}$ such that $x\in int_{\tau_{SI_{2}}(X)}\uparrow\! F\subseteq\uparrow\! F\subseteq X\mathop{\backslash}\downarrow\! y$. Then $F\in\mathcal{F}$ and $y\notin\uparrow\! F$. Thus $\uparrow\! x=\bigcap\limits_{F\in\mathcal{F}}\uparrow\! F$. Now we show that $\mathcal{F}$ is directed. Let $F_1$, $F_2\in\mathcal{F}$, we have $x\in int_{\tau_{SI_{2}}(X)}\uparrow\! F_1\cap int_{\tau_{SI_{2}}(X)}\uparrow\! F_2\in\tau_{SI_{2}}(X)$. Again from (2), there exists $F\in X^{(<\omega)}$ such that $x\in int_{\tau_{SI_{2}}(X)}\uparrow\! F\subseteq\uparrow\! F\subseteq int_{\tau_{SI_{2}}(X)}\uparrow\! F_1\cap int_{\tau_{SI_{2}}(X)}\uparrow\! F_2\subseteq \uparrow\! F_1\cap \uparrow\! F_2$. Thus we have that $\mathcal{F}$ is directed and $F\in \mathcal{F}$. By the definition of $\tau_{SI_{2}}(X)$, $F\ll_{r} x$ for all $F\in \mathcal{F}$, and then $\mathcal{F}\subseteq w(x)$. By Lemma \ref{lemma-3.2}, it is not difficult to show that $w(x)$ is directed. So $\uparrow\! x\subseteq\bigcap\limits_{F\in w(x)}\uparrow\! F\subseteq\bigcap\limits_{F\in\mathcal{F}}\uparrow\! F=\uparrow\! x$.

$(2)\Leftrightarrow (3)$: It follows from Lemma 2.2 of \cite{XY2009}.
\end{proof}

\section{$\mathcal{GD}$-convergence in $SI_2$-quasicontinuous spaces }

In this section, the concept of $\mathcal{GD}$-convergence in a poset is introduced. It is proved that the $T_{0}$ space $X$ is $SI_2$-quasicontinuous if and only if the $\mathcal{GD}$-convergence in $X$ is topological.
\begin{definition}\label{definition-4.1}
Let $X$ be a $T_{0}$ space and $(x_{j})_{j\in J}$ a net in $X$. $F\subseteq X$ is called a quasi-eventual lower bound of a net $(x_{j})_{j\in J}$ in $P$, if $F$ is finite and there exists $k\in J$ such that $x_{j}\in \uparrow F$ for all $j\geq k$ with respect to the specialization order.
\end{definition}
\begin{definition}\label{definition-4.2}
Let $X$ be a $T_{0}$ space and $(x_{j})_{j\in J}$ a net. We say $x\in X$ is a $\mathcal{GD}$-limit of the net $(x_{j})_{j\in J}$ if there exists a directed family $\mathcal{F}=\{F\subseteq X: F$ is finite\} of quasi-eventual lower bounds of the net $(x_{j})_{j\in J}$ in $X$ with respect to the specialization order such that $\bigcap_{F\in \mathcal{F}}\uparrow F\subseteq \uparrow x$. In this case we write $x\equiv_{\mathcal{GD}}$lim $x_{j}$.
\end{definition}
\begin{lemma}\label{lemma-4.3}
Let $X$ be a $T_{0}$ space and $(x_{j})_{j\in J}$ a net in $X$. If $x\equiv_{\mathcal{D}}$lim $x_{j}$ then $x\equiv_{\mathcal{GD}}$lim $x_{j}$.
\end{lemma}
\begin{proof}
 Let $X$ be a $T_{0}$ space and $(x_{i})_{i\in I}$ a net with $x\equiv_{\mathcal{D}}$lim $x_{j}$. Then there is a directed set $D$ of an eventual lower bounds of a net $(x_{j})_{j\in J}$ with $x\in D^{\delta}$. Let $\mathcal{F}=\{\{d\}: d\in D\}$, then $\mathcal{F}$ is a directed family of a quasi-eventual lower bounds of a net $(x_{j})_{j\in J}$ and $D^{\uparrow}=\bigcap\{\uparrow d: d\in D\}\subseteq \uparrow x$. Thus $x\equiv_{\mathcal{GD}}$lim $x_{j}$.
\end{proof}
\begin{remark}\label{remark-4.4}
Let $X$ be a $T_{0}$ space and $(x_{j})_{j\in J}$ a net in $X$. If $x\equiv_{\mathcal{GD}}$lim $x_{j}$ then we may not have $x\equiv_{\mathcal{D}}$lim $x_{j}$.
\end{remark}

\begin{example}\label{example-4.5}
Let $P=\mathbb{N}\cup \{\top, z\}$, where $\mathbb{N}$ denotes the set of all natural numbers. The order $\leq$ on $P$ is defined as follows \rm(see Figure 1):
\begin{enumerate}
\item[{\rm(1)}] $\forall x\in P, x\leq \top$;
\item[{\rm(2)}] $\forall x, y\in \mathbb{N}, x\leq y $ if $x$ is less than or equal to $y$ according to the usual order on natural numbers.
\end{enumerate}
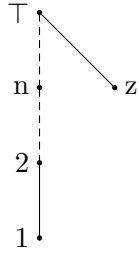
\begin{figure}[H]
\begin{center}
\begin{tikzpicture}[scale=1.0]
\path (0,1)   node[left] {$1$} coordinate (a);
\fill (a) circle (1pt);
\path (0,2)   node[left] {2} coordinate (b);
\fill (b) circle (1pt);
\path (0,3)   node[left] {n} coordinate (c);
\fill (c) circle (1pt);
\path (1,3)   node[right] {z} coordinate (d);
\fill (d) circle (1pt);
\path (0,4)  node[left] {$\top$} coordinate (e);
\fill (e) circle (1pt);
\draw  (a) -- (b)   (d) -- (e);
\draw[densely dashed] (c)--(e) (b) -- (c);
\end{tikzpicture}
\caption{A $\mathcal{GD}$-limit of a net for which is not a $\mathcal{D}$-limit.}\label{Fig:1}
\end{center}
\end{figure}

Then $(P, \alpha(P))$ is $SI_2$-quasicontinuous but not $SI_2$-continuous. Also for all $n\in \mathbb{N}, \{z, n\}\ll z$ and $\uparrow z=\bigcap_{n\in \mathbb{N}}\uparrow\{z, n\}$. Let $x_{2n}=n, x_{2n+1}=z$, then $(x_{j})_{j\in \mathbb{N}}$ is a net and $\{z, n\}$ is a quasi-eventual lower bound of it. Hence $z\equiv_{\mathcal{GD}}$lim $x_{n}$. It is not difficult to check that $z\leq x_{n}$ does not hold eventually. Thus $z$ is not a $\mathcal{D}$-limit of the net $(x_{n})_{n\in \mathbb{N}}$.
\end{example}

\begin{proposition}\label{proposition-4.1}
Let $\mathcal{F}$ be a directed family of nonempty finite sets in a $T_{0}$ space $X$. If $x\in U\in\tau_{SI_{2}}(X)$ and $\bigcap_{F\in \mathcal{F}}\uparrow F\subseteq \uparrow x$, then $F\subseteq U$ for some $F\in \mathcal{F}$.
\end{proposition}
\begin{proof}
 Suppose not, then the collection $\{F\backslash U: F\in \mathcal{F}\}$ is a directed family of nonempty finite sets.  By Lemma \ref{lem:Rudin}, there is some directed set $D\subseteq\bigcup\{F\backslash U: F\in \mathcal{F}\}$ such that $D\cap (F\backslash U)\neq\emptyset$ for all $F\in \mathcal{F}$. Then $D^{\uparrow}=\bigcap_{d\in D}\uparrow d\subseteq\bigcap_{F\in \mathcal{F}}\uparrow(F\backslash U)\subseteq\bigcap_{F\in \mathcal{F}}\uparrow F\subseteq\uparrow x$. Thus $x\in (D^{\uparrow})^{\downarrow}=D^{\delta}$. Now we have $x\in D^{\delta}\cap U\neq\emptyset$, and hence $D\cap U\neq\emptyset$ by the definition of the weakly irreducibly open set. This implies $d\in U$ for some $d\in D$, which contradicts the fact that $d$ belongs to some $F\backslash U$ which is disjoint from $U$.
\end{proof}

Now we derive the $\mathcal{GD}$-convergence in the  $SI_{2}$-quasicontinuous space is topological.

\begin{proposition}\label{proposition-4.7}
Let $X$ be an $SI_{2}$-quasicontinuous space. Then $x\equiv_{\mathcal{GD}}$lim $x_{j}$ if and only if the net $(x_{j})_{j\in J}$ converges to the element $x$ with respect to the topology $\tau_{SI_{2}}(X)$. That is, the $\mathcal{GD}$-convergence is topological.
\end{proposition}
\begin{proof}
Let $X$ be $SI_{2}$-quasicontinuous. Suppose first that $x\equiv_{\mathcal{GD}}$lim $x_{j}$ and $x\in U\in\tau_{SI_{2}}(X)$. Then there exists a directed family $\mathcal{F}=\{F\subseteq X: F$ is finite\} of a quasi-eventual lower bounds of the net $(x_{j})_{j\in J}$ in $X$ with respect to the specialization order such that $\bigcap_{F\in \mathcal{F}}\uparrow F\subseteq \uparrow x$. By Proposition \ref{proposition-4.1}, there exists  some $F\in \mathcal{F}$ such that $F\subseteq U$. Thus there exists $j_{0}\in J$ such that $x_{j}\in \uparrow F\subseteq U$ for all $j\geq j_{0}$. This shows that the net $(x_{j})_{j\in J}$ converges to the element $x$ with respect to the topology $\tau_{SI_{2}}(X)$. Conversely, we suppose that the net $(x_{j})_{j\in J}$ converges to an element $x$ with respect to the topology $\tau_{SI_{2}}(X)$. Since $X$ is $SI_2$-quasicontinuous, there exists a directed family $w(x)=\{F\subseteq X:F\in X^{(<\omega)}$ and $F\ll_{r} x\}$ and $\uparrow\! x=\bigcap\{\uparrow\! F:F\in w(x)$ \}. For all $F\in w(x)$, let $U_{F}=\{y\in X: F\ll_{r} y\}$. Then $x\in U_{F}$ and $U_{F}\in \tau_{SI_{2}}(X)$ by Proposition \ref{proposition-3.2}, and hence $x_{j}\in U_{F}$ eventually holds, that is to say, there exists $j_{0}\in J$ such that $F\ll_{r} x_{j}$ for all $j\geq j_{0}$ which implies $x_{j}\in\uparrow F$. Moreover since $w(x)$ is directed, and then one has $x\equiv_{\mathcal{GD}}$lim $x_{j}$.
\end{proof}
The converse is also true.
\begin{proposition}\label{proposition-4.8}
Let $(X, \tau)$ be a $T_{0}$ space. If the $\mathcal{GD}$-convergence with respect to the topology $\tau_{SI_{2}}(X)$ is topological, then $X$ is SI$_{2}$-quasicontinuous.
\end{proposition}
\begin{proof}
Suppose that the $\mathcal{GD}$-convergence with respect to the topology $\tau_{SI_{2}}(X)$ is topological.
For any $x\in X$, let $J=\{(U, n, a)\in N(x)\times\mathbb{N}\times X: a\in U\}$, where $N(x)$ consists of all weakly irreducibly open sets which contain $x$, and define an order on $J$ as follows: $(U, m, a)\leq (V, n, b)$ if and only if $V$ is proper subset of $U$ or $U=V$ and $m\leq n$. Obviously, $J$ is directed. Let $x_{j}=a$ for all $j=(U, m, a)\in J$. Then it is not difficult to check that the net $(x_{j})_{j\in J}$ converges to the element $x$ with respect to the weakly irreducible topology $\tau_{SI_{2}}(X)$, and hence $x\equiv_{\mathcal{GD}}$lim $x_{j}$. Thus it concludes that there is a directed family $\mathcal{F}=\{ F\subseteq X: F$ is finite\} of a quasi-eventual lower bounds of the net $(x_{j})_{j\in J}$ in $X$ with respect to the specialization order such that $\bigcap_{F\in \mathcal{F}}\uparrow F\subseteq \uparrow x$. Now we prove that (1) for all $F\in \mathcal{F}, F\ll_{r} x$; $(2) \bigcap_{F\in \mathcal{F}}\uparrow F= \uparrow x$.

(1) Let $E\subseteq X$ be irreducible with $x\in E^{\delta}$. Since $F$ is a quasi-eventual lower bound of the net $(x_{j})_{j\in J}$, there is $j_{0}=(U, m, a)\in J$ such that $x_{j}\in \uparrow F$ for all $j=(V, n, b)\geq j_{0}$. Notice $x\in U$, so $E^{\delta}\cap U\neq\emptyset$, and then $E\cap U\neq\emptyset$. Pick $e\in E\cap U$. Set $i=(U, m+1, e)$, then $i\geq (U, m, a)=j_{0}$. Thus $e=x_{i}\in \uparrow F$, that is, $F\ll_{r} x$.

(2) We only need to show that $\uparrow x\subseteq\bigcap_{F\in \mathcal{F}}\uparrow F$. Suppose not, then there exists $y\geq x$ but $y\notin\bigcap_{F\in \mathcal{F}}\uparrow F$, that is, there exists $F\in \mathcal{F}$ with $y\notin \uparrow F$. And then $\uparrow F\subseteq X\backslash \downarrow x$. Again since $F$ is a quasi-eventual lower bound of the net $(x_{j})_{j\in J}$, there exists $j_{0}=(U, m, a)\in J$ such that $x_{j}\in \uparrow F$ for all $j=(V, n, b)\geq j_{0}$. Now we have $x\in U$. Set $i=(U, m+1, x)$, then $i\geq (U, m, x)=j_{0}$. Thus $x=x_{i}\in \uparrow F\subseteq X\backslash \downarrow x$, a contradiction.
Finally, it claims that $\Uparrow_{r} H\in\tau$ for all nonempty finite sets $H\subseteq X$. Indeed, if $y\in\Uparrow_{r} H$, then there exists a nonempty finite set $K\subseteq X$ such that $H\ll_{r} K\ll_{r} y$ by Theorem \ref{theorem-3.1}. From the above argument, as long as we replace $J=\{(U, n, a)\in N(x)\times\mathbb{N}\times X: a\in U\}$ with $I=\{(U, n, a)\in N(y)\times\mathbb{N}\times X: a\in U\}$, where $N(y)$ consists of all weakly irreducibly open sets containing $y$,  similarly we have that $(x_{i})_{i\in I}$ converges to $y$ with respect to the topology $\tau_{SI_{2}}(X)$, and then there exists directed family $\mathcal{G}=\{ G\subseteq X: G$ is finite\} of a quasi-eventual lower bounds of the net $(x_{j})_{j\in J}$ in $X$ with respect to the specialization order such that $\bigcap_{G\in \mathcal{G}}\uparrow G\subseteq \uparrow y$. Note that $K\ll_{r} y$. By Lemma \ref{lemma-3.2}, there exists $G\in \mathcal{G}$ such that $G\subseteq\uparrow K$. By the convergence of the net, there exists $i_{0}=(U_{1}, m, a)$ such that $x_{i}=b\in\uparrow G$ for all $i\geq i_{0}$, where $U_{1}\in \tau_{SI_{2}}(X)$. In particular, we have $(U_{1}, m+1, u)\geq (U_{1}, m, a)=i_{0}$ for all $u\in U_{1}$. We conclude that $u\in \uparrow G$, and hence $U_{1}\subseteq\uparrow G$. Next we show that $U_{1}\subseteq\Uparrow_{r} H$. Let $u\in U_{1}$ and $E\in Irr(X)$ with $u\in E^{\delta}\cap U_{1}$. Since $U_{1}$ is weakly irreducible open, then $E\cap U_{1}\neq\emptyset$. Pick $e\in E\cap U_{1}$. And then we have $e\in E\cap U_{1}\subseteq E\cap\uparrow G\subseteq E\cap\uparrow K\subseteq E\cap\uparrow H$. Thus we have $H\ll_{r} u$, and hence we have $y\in U_{1}\subseteq\Uparrow_{r} H$. This shows that $\Uparrow_{r} H\in\tau$. Hence $X$ is $SI_{2}$-quasicontinuous.
\end{proof}
From Proposition \ref{proposition-4.7} and Proposition \ref{proposition-4.8} we have:

\begin{theorem}\label{theorem:-4.9}
Let $X$ be a $T_{0}$ space. Then the following statements are equivalent:
\begin{enumerate}
\item[\rm(1)] $SI_{2}$-quasicontinuous space;

\item[\rm(2)] The $\mathcal{GD}$-convergence with respect to the topology $\tau_{SI_{2}}(X)$ is topological, that is, for all $x\in X$ and all nets $(x_{j})_{j\in J}$ in $X$, $x\equiv_{\mathcal{GD}}$lim $x_{j}$ if and only if $(x_{j})_{j\in J}$ converges to $x$ with respect to the weakly irreducible topology.
\end{enumerate}
\end{theorem}
\begin{corollary}[\cite{ZL2013}]
Let $P$ be a dcpo. Then the following conditions are equivalent:
\begin{enumerate}
\item[\rm(1)] $P$ is a quasicontinuous domain;
\item[\rm(2)] $S^{\ast}$-convergence in $P$ is topological for the Scott topology $\sigma(P)$, that is, for all $x\in P$ and all nets $(x_{j})_{j\in J}$ in $P$, $(x_{j})_{j\in J}$ $S^{\ast}$-converges to $x$ if and only if $(x_{j})_{j\in J}$ converges to $x$ with respect to the Scott topology.

\end{enumerate}
\end{corollary}

\begin{corollary}[\cite{R2016}]
Let $P$ be a poset. Then the following conditions are equivalent:
\begin{enumerate}
\item[\rm(1)] $P$ is $s_2$-quasicontinuous;

\item[\rm(2)] The $\mathcal{GS}$-convergence in $P$ is topological for the weak Scott topology $\sigma_2(P)$, that is, for all $x\in P$ and all nets $(x_{j})_{j\in J}$ in $P$, $x\equiv_{\mathcal{GS}}$lim $x_{j}$ if and only if $(x_{j})_{j\in J}$ converges to $x$ with respect to the weak Scott topology.
\end{enumerate}
\end{corollary}

\section{Conclusions}
In this paper, we introduce the concepts of $SI_{2}$-quasicontinuous spaces and present one way to generalize $\mathcal{D}$-convergence of nets for arbitrary topological spaces by the cuts. Some convergence theoretical characterizations of $SI_{2}$-quasicontinuity of spaces are given and come to the main conclusions are: (1) A space is $SI_{2}$-quasicontinuous if and only if its weakly irreducible topology is hypercontinuous under inclusion order; (2) A $T_{0}$ space $X$ is $SI_{2}$-quasicontinuous if and only if the $\mathcal{GD}$-convergence in $X$ is topological.

\section*{Acknowledgement}
We would like to thank the referees for the numerous and helpful suggestions that have improved this paper substantially.

\end{document}